\documentclass{amsart}

\usepackage{amssymb,amsbsy,amsthm,amsmath,graphicx,epsfig,color}

\newtheorem{thm}{Theorem}[section]
\newtheorem*{thm*}{Theorem}

\newtheorem{lemma}[thm]{Lemma}

\newtheorem{proposition}[thm]{Proposition}

\newtheorem{corollary}[thm]{Corollary}

\theoremstyle{definition}

\newtheorem*{dfn*}{Definition}
\newtheorem{definition}[thm]{Definition}

\newtheorem*{ques*}{Questions}

\newtheorem{remark}[thm]{Remark}

\newcommand{\T}{\mathbb{T}}

\newcommand{\DD}{\mathcal{D}}

\newcommand{\A}{\mathcal{A}}

\newcommand{\N}{\mathbb{N}}

\newcommand{\C}{\mathbb{C}}

\newcommand{\Z}{\mathbb{Z}}

\newcommand{\veps}{\varepsilon}

\newcommand{\aveN}{\frac{1}{N}\sum_{n=1}^N}

\newcommand{\limaveN}{\lim_{N\to\infty} \aveN}

\newcommand{\ol}[1]{\overline{#1}}

\newcommand{\orb}{\text{orb}}

\newcommand{\la}{\langle}
\newcommand{\ra}{\rangle}

\newcommand{\jdlg}{\text{Jacobs-de Leeuw-Glicksberg}}
\newcommand{\jdlgd}{\text{Jacobs-de Leeuw-Glicksberg decomposition}}
\newcommand{\Ell}{L}
\newcommand{\mps}{\text{measure-preserving system}}
\newcommand{\mpss}{\text{measure-preserving systems}}
\newcommand{\EY}{E_Y}
\newcommand{\fix}{\text{Fix}}
\newcommand{\Fix}{\text{Fix}}

\newcommand{\rg}{\text{rg}}
\newcommand{\limN}{\lim_{N\to\infty}}
\newcommand{\Eins}{\mathbf{1}}
\newcommand{\clin}{\ol{\text{lin}}}
\newcommand{\kr}{\text{kr}}

\newcommand{\LLL}{{L}}

\title{A view on multiple recurrence}

\author{Tanja Eisner}
\address{Institute of Mathematics, University of Leipzig,
P.O. Box 100 920, 04009 Leipzig, Germany}
\email{eisner@math.uni-leipzig.de}


\begin{document}

\maketitle

\begin{center}
{\emph{\small Dedicated to Jacob (Jaap) Korevaar\\ on the occasion of his approaching 100$^{\text{th}}$ birthday}}
\end{center}

\begin{abstract}
In this note we present a proof of multiple recurrence for ergodic systems (and thereby of Szemer\'edi's theorem) being a mixture of three known proofs. It is based on a conditional version of the Jacobs-de Leeuw-Glicksberg decomposition and  properties of the Gowers-Host-Kra uniformity seminorms. 
\end{abstract}

%
%
%

\section{Introduction}

The celebrated Szemer\'edi theorem \cite{Szem} from 1975 asserts that every subset of the natural numbers with positive upper density 
contains arithmetic progressions of arbitrary length. This means that as soon as a set does not vanish asymptotically it has certain structure inside. 
In 1977 Furstenberg \cite{Fu77} presented his groundbreaking ergodic theoretic proof of Szemer\'edi's theorem  based on multiple recurrence 
which has had enormous impact. 
In particular, various related results in additive number theory were proven using ergodic theoretic methods such as the Furstenberg-Katznelson multidimensional Szemer\'edi theorem \cite{FuKa78}, the Bergelson-Leibman polynomial Szemer\'edi theorem  \cite{BL96}, the Green-Tao theorem on the existence of arithmetic progressions in the primes \cite{GT08,GT10}, the Tao-Ziegler polynomial Green-Tao theorem \cite{TZ08,TZ13,TZ18} and the multidimensional Green-Tao theorem \cite{TaZi15,CMT,FoZh} (where \cite{CMT} does not use ergodic theory).  For some related works see also \cite{BM00,FHK07,BLL08,FrWi09,WoZi12,FHK13,Frantz15,MR19,BMR20}.

In this note we present a proof of multiple recurrence for ergodic systems (and hence of Szemer\'edi's theorem) being a mixture of three known proofs, namely Tao's modification \cite{Tao-ET} of the classical Furstenberg-Katznelson proof \cite{FuKa78}, the ori\-ginal Furstenberg proof \cite{Fu77} and (the beginning of)
the proof of multiple convergence by Host and Kra \cite{HK}. 
For a finitary quantitative proof of Szemer\'edi's theorem with similar philosophy see Tao \cite{Tao-szem}.

Like the original Furstenberg proof \cite{Fu77} (and in contrast to the classical Furstenberg-Katznelson proof), our proof 
does not use transfinite induction  and deals for multiple recurrence of order $k$ with a tower of  $k-1$ factors, namely Furstenberg's distal factors
of order 
less than or equal to $k-1$. 
We give an alternative proof of Furstenberg's result  \cite{Fu77} that the distal factor of order $k-1$ is characteristic for $k$-term multiple ergodic averages by showing that this factor is an extension of the Host-Kra-Ziegler factor of order $k-1$ (via working with the corresponding subspaces of functions). 
 Thereby we do not need any information on the Host-Kra-Ziegler factors other than
the description of the corresponding subspaces via the Gowers-Host-Kra seminorms, not even the fact that they are factors. Note that the original proof of Furstenberg of the characteristic property of the distal factors of finite order is based on a description of conditional eigen\-functions\footnote{in Furstenberg's terminology, generalized eigenfunctions} of fibered product systems (see \cite[Theorem 7.1]{Fu77}) and analysis of diagonal measures, with a subsequent simplification by Frantzikinakis \cite[Theorem 5.2]{Frantz04}, still based on \cite[Theorem 7.1]{Fu77}, using the van der Corput trick.

Two decompositions will play a central role, namely a conditional version of the \jdlgd\ and the Host-Kra decomposition. 
The first one, decomposing the $\Ell^2$-space into conditionally almost periodic and conditionally weakly mixing functions, is a stronger version due to Tao \cite[Chapter 2]{Tao-ET} of the dichotomy between almost periodic and weakly mixing extensions due to Furstenberg, Katznelson \cite{FuKa78}. 
The Host-Kra decomposition decomposes, for every $k$, the $\Ell^2$-space into a part where the Gowers-Host-Kra seminorm of order $k+1$ vanishes and the orthogonal complement known as the Host-Kra-Ziegler factor of order $k$. 
Then properties of the Gowers-Host-Kra seminorms and the conditional Jacobs-de Leeuw-Glicksberg decomposition quickly imply 
 that the distal factors of finite order are extensions of the Host-Kra-Ziegler factors of the corresponding order,
see Proposition \ref{prop:Z_k-ap} below.  
(Host, Kra \cite[Lemma 6.2]{HK}, \cite[Lemma 18.2]{HK-book} and Ziegler \cite[Thm.~6.1]{Ziegler} showed\footnote{Host and Kra showed that the Host-Kra-Ziegler factors are isometric extensions (as defined by Furstenberg \cite{Fu77}) of each other. See, e.g., Glasner \cite[Theorem 9.14]{Glasner-book} for the characterization of isometric extensions as those being generated by conditional eigenfunctions  implying almost periodicity of such extensions.} a stronger property of the Host-Kra-Ziegler factors, namely that they are compact extensions of each other. For an alternative proof using the conditional Jacobs-de Leeuw-Glicksberg decomposition see Zorin-Kranich \cite[Lemma 10.1]{Pavel-ET}.)

The rest of the argument is standard. By definition, the distal factors of finite order are maximal compact extensions of each other. By induction, starting with the fixed (one-point) factor and using the fact that compact extensions preserve multiple recurrence (Proposition \ref{prop:ap-ext-MR} below), one has multiple recurrence for the distal factors of all orders. Since the distal factor of order $k-1$ is an extension of the Host-Kra-Ziegler factor of order $k-1$, it
is characteristic for $k$-term multiple ergodic averages by the generalized von Neumann theorem (Proposition \ref{prop:ghk-control-mult-ave} and Corollary \ref{cor:ghk-control-mult-ave} below), 
completing the proof.

As indicated above, the structure of the proof 
 is essentially implicit in Host, Kra \cite{HK}, Ziegler \cite{Ziegler} as well as in Furstenberg \cite{Fu77} combined with Furstenberg, Katznelson, Ornstein \cite{FuKaOr82}, cf.~Tao \cite[Remark 4.7]{Tao-szem}. We present here a view based on the conditional Jacobs-de Leeuw-Glicksberg decomposition and properties of the Gowers-Host-Kra seminorms.

\textbf{Acknowledgements.} The author is very grateful to B\'alint Farkas and Asgar Jamneshan for helpful discussions, comments and references.  

\section{Preliminaries}


Throughout this article we call a triple $(X,\mu,\varphi)$ a \emph{measure-preserving system}\ if $(X,\mu)$ is a 
probability space and $\varphi$ is a measurable invertible transformation on $X$ with measurable inverse which is $\mu$-preserving, i.e., $\mu(\varphi^{-1}A)=\mu(A)$ holds for every measurable set $A\subset X$.\footnote{In this case $\varphi^{-1}$ is clearly $\mu$-preserving as well. The invertibility assumption is technical and will assure that the conditional expectation commutes with the transformation. Since in our context one can work with invertible systems without loss of generality, we assume all systems to be invertible.}
For a \mps\ $(X,\mu,\varphi)$ we call the linear invertible isometry $T:\Ell^1(X,\mu)\to \Ell^1(X,\mu)$ defined by $Tf:=f\circ \varphi$ the \emph{Koopman operator} of the system. 
The Koopman operator acts as  an invertible isometry on $\Ell^p(X,\mu)$ for every $p\in [1,\infty]$ and is unitary on $\Ell^2(X,\mu)$. Since we mostly work with the Koopman operator $T$ instead of the transformation $\varphi$, we often write $(X,\mu,T)$ instead of $(X,\mu,\varphi)$.

A measure-preserving system  $(X,\mu,\varphi)$ is called \emph{ergodic} if it has no non-trivial invariant sets, i.e., if every measurable set satisfying $\varphi^{-1}A\subset A$ up to a null set (which is equivalent to $\varphi^{-1}A=A$ and hence, by the invertibility of $\varphi$, to $A=\varphi(A)$ up to a null set) 
has measure zero or one. For the corresponding Koopman operator $T$ on every $\Ell^p(X,\mu)$ this is equivalent to $\Fix T=\C\Eins$ (i.e., $T$  has only constant invariant functions). 

We now define multiple recurrence and multiple convergence. Thereby,  $f>0$ means $f\geq 0$ and $f\neq 0$ (as element of $\Ell^\infty(X,\mu)$).

\begin{definition}
Let $(X,\mu,T)$ be a measure-preserving system and $k\in \N$. We say that $(X,\mu,T)$ satisfies 
\begin{itemize}
\item \emph{multiple recurrence of order $k$} (shortly \emph{MR$_k$}) if for every $0<f\in\Ell^\infty(X,\mu)$ 
\begin{equation}\label{eq:mult-rec}
\liminf_{N\to\infty}\aveN \int_X f\cdot T^nf\cdot T^{2n}f \cdots T^{kn}f\,d\mu>0.
\end{equation}
\item \emph{multiple convergence of order $k$} 
if for every $f_1,\ldots,f_k\in\Ell^\infty(X,\mu)$ the limit of multiple ergodic averages 
\begin{equation}\label{eq:mult-erg-ave}
\aveN T^nf_1\cdot T^{2n}f_2 \cdots T^{kn}f_k
\end{equation}
exists in $\Ell^2(X,\mu)$.
\end{itemize}
We say that $(X,\mu,T)$ satisfies \emph{multiple recurrence} (shortly \emph{MR}) if it satisfies MR$_k$ for all $k\in\N$ and define analogously \emph{multiple convergence}.
\end{definition}
Note that, given multiple convergence of order $k$, the property MR$_k$ means that every function $f\in \Ell^\infty(X,\mu)$ with $f>0$ correlates with the limit of the averages $\aveN T^nf\cdots T^{kn}f$. Moreover, for functions of the form $f=\Eins_A$ for some $A\subset X$ with $\mu(A)>0$ property \eqref{eq:mult-rec} implies for the underlying transformation $\varphi$ that $\mu(A\cap \varphi^{-n}A\cap\ldots\cap \varphi^{-kn}A)>0$ holds for many $n\in\N$ explaining the name ``multiple recurrence''.

We now state Szemer\'edi's theorem which inspired the study of multiple recurrence and multiple convergence. 

\begin{thm}[Szemer\'edi \cite{Szem}]
Let $C\subset \N$ have \emph{positive upper density}, i.e., satisfy $\ol{d}(C):=
\limsup_{N\to\infty} \frac{|C\cap\{1,\ldots,N\}|}N>0$. Then $C$ contains arbitrarily long arithmetic progressions, i.e., for every $k\in\N$ there exist 
$a,n\in \N$ with $a,a+n,\ldots,a+kn\in C$. 
\end{thm}

Furstenberg  \cite{Fu77} divided his ergodic theoretic proof of Szemer\'edi's theorem into two parts. The difficult part was to show multiple recurrence for ergodic systems. 

\begin{thm}[Furstenberg, multiple recurrence for ergodic systems]\label{thm:mr}
Every 
ergodic\footnote{For our purposes ergodic systems  suffice by Proposition \ref{prop:corr-principle}. Note that multiple recurrence holds for general (also not necessarily invertible) \mpss, see, e.g., Einsiedler, Ward \cite[pp.~177--178]{EW}.} 
\mps\ 
satisfies MR.\footnote{Furstenberg  \cite{Fu77} showed a stronger property than MR where the lower Ces\'aro limit $\liminf_{N\to\infty}\aveN$ in \eqref{eq:mult-rec} is replaced by the lower Banach limit $\liminf_{N-M\to\infty}\sum_{n=M+1}^N$.}  
\end{thm}

The easy part was to build the following bridge between number theory and ergodic theory, see, e.g., Furstenberg \cite[p.~77]{Fu-book}, Kra \cite[Section 2.2]{Kra} or \cite[Section 20.1]{EFHN} for details. 

\begin{proposition}[Furstenberg, correspondence principle]\label{prop:corr-principle}
Multiple recurrence for ergodic systems implies the assertion of Szemer\'edi's theorem. More precisely, 
for every $C\subset \N$ there exists an ergodic \mps\ $(X,\mu,\varphi)$ and a measurable set $A\subset X$ with $\mu(A)\geq \ol{d}(C)$\footnote{By the same argument one can easily replace the upper density $\ol{d}(C)$ by the upper Banach density $d^*(C):=
\limsup_{N-M\to\infty}\frac{|C\cap \{M+1,\ldots,N\}|}{N-M}$ leading to the corresponding stronger version of Szemer\'edi's theorem proved by Furstenberg \cite{Fu77}.} such that, for every $k\in\N$, the set $C$ contains an arithmetic progression of length $k+1$ if and only if there exists $n\in\N$ such that the set 
$A\cap \varphi^{-n}A\cap \ldots\cap \varphi^{-kn}A$ has positive measure.
\end{proposition}

In particular, it suffices to prove 
that for every $f\in\Ell^\infty(X,\mu)$ with $f>0$ there exists $n\in\N$ such that $\int_X f\cdot T^nf\cdot T^{2n}f \cdots T^{kn}f\,d\mu>0$ which is a formally weaker property but turns out to be equivalent to MR$_k$. For a discussion and a deduction of multiple recurrence back from Szemer\'edi's theorem see, e.g., \cite[Section 20.2]{EFHN}.

Multiple convergence remained open 
 until it was proved by Host, Kra \cite{HK} in 2005, with a subsequent alternative proof by Ziegler  \cite{Ziegler}.

\begin{thm}[Multiple convergence, Host, Kra \cite{HK}, Ziegler \cite{Ziegler}]\label{thm:mr}
Every 
measure-preserving system 
satisfies multiple convergence. 
\end{thm}

\medskip

We now collect some definitions. 

We do not need to assume systems to be separable (i.e., the underlying $\sigma$-algebra to be countably generated up to null sets). 
So we work with the following abstract notion of a factor. A \mps\ $(Y,\nu,S)$ is called a \emph{(Markov) factor} of $(X,\mu,T)$ if there exists a Markov homomorphism\footnote{An operator $J:\Ell^1(Y,\nu)\to \Ell^1(X,\mu)$ is called \emph{Markov} if it is positive (i.e., $f\geq 0$ implies $Jf\geq 0$) and satisfies $J\Eins=\Eins$ as well as $ J'\Eins=\Eins$ (i.e., $\int_X Jf\,d\mu=\int_Yf\,d\nu$ holds for every $f\in \Ell^1(Y,\nu)$). A Markov operator is called \emph{Markov homomorphism} (or \emph{Markov embedding}) if it preserves the absolut value or, equivalently, is multiplicative on $\Ell^\infty(Y,\nu)$, see, e.g., \cite[Theorem 7.23]{EFHN}.}  $J:\Ell^1(Y,\nu)\to \Ell^1(X,\mu)$ which respects the Koopman operators, i.e., satisfies $JS=TJ$. In this case $(X,\mu,T)$ is called a  \emph{(Markov) extension} of $(Y,\nu,S)$ and $J$ is called the \emph{(Markov) factor map}.  


\begin{remark}\label{rem:factors}
\begin{itemize}
\item[(a)] For a (Markov) factor $(Y,\nu,S)$ of a measure-preserving system $(X,\mu,T)$ with  factor map $J:\Ell^1(Y,\nu)\to \Ell^1(X,\nu)$ the set $J\Ell^\infty(Y,\nu)$ is a conjugation invariant, $T$- and $T^{-1}$-invariant subalgebra of $\Ell^\infty(X,\mu)$ containing $\Eins$, and there is a (up to isomorphism one-to-one) correspondence between  (Markov) factors of $(X,\mu,T)$ and  conjugation invariant, $T$- and $T^{-1}$-invariant subalgebras of $\Ell^\infty(X,\mu)$ containing $\Eins$. Every such subalgebra $\A$ of functions further corresponds to a $\varphi$- and $\varphi^{-1}$-invariant sub-$\sigma$-algebra of the original $\sigma$-algebra on $X$ for the underlying transformation $\varphi$, namely the smallest $\varphi$- and $\varphi^{-1}$-invariant sub-$\sigma$-algebra such that all functions from $\A$ are measurable with respect to it. We will work with the above characterization of factors via subalgebras of functions. 
\item[(b)]
For separable systems every (Markov) factor map is induced by a \emph{point factor map}, i.e., $J$ is of the form $Jf=f\circ \pi$ for a  morphism $\pi:X\to Y$, i.e., a measure-preserving map such that $\pi\circ \varphi=\psi \circ \pi$ holds $\mu$-a.e.~for the underlying transformations $\varphi$ and $\psi$ on $X$ and $Y$, respectively.
For details see, e.g., \cite[Chapter 7]{EF-book}. The readers who prefer point factors to Markov factors can assume without loss of generality that the system is separable by passing to the $\varphi$- and $\varphi^{-1}$-invariant sub-$\sigma$-algebra generated by the function $f$ for multiple recurrence and by the functions $f_1,\ldots,f_k$ for multiple convergence of order $k$, respectively.
\end{itemize}
\end{remark} 

In particular, for a \mps\ $(X,\mu,T)$ with factor $(Y,\nu,S)$ the corresponding factor map $J:\Ell^1(Y,\mu)\to \Ell^1(X,\nu)$ 
 acts as a contraction w.r.t.~the $\Ell^p$-norm for every $p\in[1,\infty]$. 
The adjoint operator $J'$ extends to a Markov operator 
$$
\EY:=J':\Ell^1(X,\mu)\to\Ell^1(Y,\nu)$$ 
called the \emph{conditional expectation operator}. It also acts as a contraction w.r.t.~the $\Ell^p$-norm for every $p\in[1,\infty]$.  Recall that the conditional expectation operator satisfies 
\begin{equation}\label{eq:EY}
\int_Y \EY f\,d\nu=\int_X f\,d\mu,
\quad
\EY(Jg\cdot f)=g \EY f\quad \forall f\in \Ell^1(X,\mu)\ \forall g\in\Ell^\infty(Y,\nu).
\end{equation}
Since we assume factors to be invertible, 
one also has
\begin{equation}\label{eq:ET=SE}
\EY T=S\EY.
\end{equation}
The operator 
$$
P_Y:=J\EY:\Ell^2(X,\mu)\to\Ell^2(X,\mu)
$$ 
is the orthogonal projection onto the subspace $J\Ell^2(Y,\nu)$. We call $P_Y$  the 
\emph{projection onto the factor} $Y$. It extends to a Markov operator acting as a contraction w.r.t.~every $\Ell^p$-norm, $p\in[1,\infty]$. Another important property of $P_Y$ is 
\begin{equation}\label{eq:PYf>0}
P_Yf>0\quad \text{for every }f>0. 
\end{equation}
(Indeed, the Markov property of $J$ and $\EY$ implies $P_Yf\geq 0$ 
and $\int_XP_Yf\,d\mu=\int_Xf\,d\mu>0$ ensuring $P_Yf\neq 0$.)

We say that a subspace $\A$ of $\Ell^2(X,\mu)$ is \emph{induced by a  factor} $(Y,\nu,S)$ of $(X,\mu,T)$  if $\A=J\Ell^2(Y,\nu)$ for the corresponding Markov embedding $J$ or, equivalently,  
$\A=\Ell^2(X,\Sigma',\mu)$ for the $\varphi$-invariant sub-$\sigma$-algebra $\Sigma'$ corresponding to $(Y,\nu,S)$ for the underlying transformation $\varphi$ on $X$.
 By Remark \ref{rem:factors}(a),  this is the case if and only if $\A\cap \Ell^\infty(X,\mu)$ is a conjugation invariant, $T$- and $T^{-1}$-invariant subalgebra of $\Ell^\infty(X,\mu)$ containing $\Eins$.

\section{Single recurrence and von Neumann's decomposition}

We recall how the mean ergodic theorem and the property MR$_1$ rely on the von Neumann decomposition 
\begin{equation}\label{eq:mul-vN-dec}
H=\fix T\oplus \ol{\rg(I-T)}
\end{equation}
for contractions $T$ on a Hilbert space $H$, see, e.g., \cite[Theorem 8.6]{EFHN}. 
Let $(X,\mu,T)$ be a \mps\ and $f\in\Ell^2(X,\mu)=:H$.\footnote{For $k=1$ one does not need $f$ to be bounded.} By the telescopic sum argument the second part of \eqref{eq:mul-vN-dec} does not contribute to the limit of the ergodic averages $\aveN T^nf$ and  $T$ acts as the identity operator on the first part of  \eqref{eq:mul-vN-dec}.  Thus if $f>0$ then
$$
\limaveN \int_X f\cdot T^nf\,d\mu=
\left\la\limaveN T^nf,f\right\ra =
\|P_{\Fix T} f\|_{\Ell^2(X,\mu)}^2>0.
$$
The last inequality follows from the measure preserving property of $T$ (which implies $\int_X P_{\Fix T}f\,d\mu=\int_X f\,d\mu\neq 0$) and reflects the fact that $\Fix T$ is induced by a factor, namely the fixed factor. 
In other words, the second part in the von Neumann decomposition \eqref{eq:mul-vN-dec} does not contribute to the limit of $\aveN T^nf$ 
and we need $P_{\Fix T}f>0$ 
to restrict to the case $\Ell^2(X)=\Fix T$ where the assertion is trivial due to the very structured behavior of the orbits (being one point).

\section{Double recurrence and the classical \jdlgd}

To illustrate the argument for higher $k$ we now briefly discuss the case $k=2$ which is analogous to the case $k=1$. For a detailed exposition see, e.g., \cite[Section 8.3]{EF-book}. 



\begin{definition}
Let $T$ be a contraction on a Hilbert space $H$. We call $f\in H$ \emph{weakly mixing} if $\limaveN |\la T^nf,f\ra|^2=0$.
\end{definition}

\begin{remark}\label{rem:wmix}
By the Koopman-von Neumann lemma, see, e.g., \cite[Lemma 9.16]{EFHN}, for a contraction $T$ on a Hilbert space $H$, $f\in H$ is weakly mixing if and only if there exists a set $J\subset \N$ with density $1$ (i.e., satisfying $d(J):=
\lim_{N\to\infty} \frac{|J\cap\{1,\ldots,N\}|}N=1$) such that $\lim_{n\to\infty,n\in J}\la T^nf,f\ra=0$.  
\end{remark}

Double convergence and double recurrence rely on the following decomposition replacing the von Neumann decomposition, where $\T$ denotes the unit circle. 

\begin{thm}[Jacobs-de Leeuw-Glicksberg decomposition]\label{thm:JdLG}
Let $T\in \LLL(H)$ be a contraction on a Hilbert space $H$. Then one has the orthogonal decomposition 
\begin{equation}\label{eq:jdlg-dec}
H=\clin\{f\in H:\, Tf=\lambda f\text{ for some }\lambda\in\T\}\oplus \{f\in H:\, f\text{ is weakly mixing}\}.
\end{equation}
\end{thm}

Let now $(X,\mu,T)$ be a \mps\ and $f,g\in \Ell^\infty(X,\mu)$. One shows via the so-called van der Corput trick that the double averages
$$
\aveN T^nf \cdot T^ng
$$
converge to zero whenever $f$ or $g$ is weakly mixing. Since the first part of the decomposition \eqref{eq:jdlg-dec} is induced by a factor, namely the Kronecker factor, we have that $P_\kr f,P_\kr g\in\Ell^\infty(X,\mu)$ for the corresponding projection $P_\kr$. Thus 
double convergence reduces to showing it for eigenfunctions for which it is immediate. For double recurrence let $g:=f>0$. Then $P_\kr f>0$ 
by \eqref{eq:PYf>0} and we can assume by \eqref{eq:jdlg-dec} and the above observation that $(X,\mu,T)$ coincides with its Kronecker factor. In this case the orbit of $f$ is compact and one uses almost periodicity of $f$ to show that the limit in \eqref{eq:mult-rec} (for $k=2$) is greater than zero. 

To summarize, the second part in the decomposition \eqref{eq:jdlg-dec} does not contribute to the limit and, using the fact that the first part is induced by a factor, we can assume that this factor coincides with the original system. Then one uses the very structured behavior of the orbits. 

Next we consider two generalizations of \eqref{eq:jdlg-dec}, the  conditional Jacobs-de Leeuw-Glicksberg decomposition and the Host-Kra decomposition, see \eqref{eq:rel-jdlg-a+w} and \eqref{eq:hk-dec} below.

\section{Conditional \jdlgd}

We now present the abstract conditional setting introduced by Tao \cite[Chapter 2]{Tao-ET}. To simplify the notation we will shorten
$\Ell^p(X,\mu)$ to $\Ell^p(X)$ for a probability space $(X,\mu)$.
Moreover, by writting $cf$ for $c\in\Ell^\infty(Y)$ and $f\in \Ell^1(X)$ we identify $c$ with its image under $J:\Ell^1(Y)\to\Ell^1(X)$.

\begin{definition}
Let $(X,\mu,T)$ be a \mps\ with 
factor $(Y,\nu,S)$. 
For $f,g\in \Ell^2(X)$ we call the function
$$
\la f,g\ra_{\Ell^2(X|Y)}:=\EY(f\cdot\ol{g})
\in\Ell^1(Y)
$$
the \emph{conditional scalar product} of $f$ and $g$ w.r.t.~$Y$ and the function
$$
\|f\|_{\Ell^2(X|Y)}:=(\la f,f\ra_{\Ell^2(X|Y)})^{\frac12}=(\EY(|f|^2))^{\frac12}\in\Ell^2(Y)
$$
the \emph{conditional 
norm} of $f$ w.r.t.~$Y$. Moreover, the \emph{conditional $\Ell^2$-space} is defined by
$$
\Ell^2(X|Y):=\{f\in\Ell^2(X):\, \|f\|_{\Ell^2(X|Y)}\in\Ell^\infty (Y)\}.
$$
A function $f\in \Ell^2(X|Y)$ is called 
\emph{conditionally almost periodic} if for every $\veps>0$ there exists a \emph{finitely generated module zonotope} $Z$ (i.e., a set of the form 
$$
\{c_1f_1+\ldots+c_mf_m:\,\|c_1\|_{\Ell^\infty(Y)},\ldots,\|c_m\|_{\Ell^\infty(Y)}\leq 1\}
$$
 for some $m\in\N$, $f_1,\ldots,f_m\in\Ell^2(X|Y)$) such that for every $n\in\Z$ there exists  $g\in Z$ with 
$$
\|\|T^nf-g\|_{\Ell^2(X|Y)}\|_{\Ell^\infty(Y)}<\veps.
$$ 
We denote the space of all conditionally almost periodic functions by $A(X|Y)$. 

A function $f\in \Ell^2(X)$ is called  \emph{conditionally weakly mixing}\footnote{We follow here Zorin-Kranich \cite[Section 8]{Pavel-ET}. Tao \cite[Section 2.14.1]{Tao-ET} calls a function $f\in\Ell^2(X|Y)$ conditionally weakly mixing if $\limaveN
\|\la T^nf,f\ra_{\Ell^2(X|Y)}\|_{\Ell^2(Y)}^2=0$. It is easy to see that these two definitions are equivalent on $\Ell^2(X|Y)$.} if 
$$
\limaveN
\|\la T^nf,f\ra_{\Ell^2(X|Y)}\|_{\Ell^1(Y)}
=0.
$$
We denote the set of all such functions by $W(X|Y)$.
\end{definition}
To shorten the notation we abbreviate 
$$
\|f\|_{\Ell^{2,\infty}(X|Y)}:=\|\|f\|_{\Ell^2(X|Y)}\|_{\Ell^\infty(Y)}
$$
for $f\in \Ell^2(X|Y)$. It is easy to see, by establishing a Cauchy-Schwarz inequality for the conditional scalar product, that $\|\cdot\|_{\Ell^{2,\infty}(X|Y)}$ is a norm on $\Ell^2(X|Y)$. 
\begin{remark}\label{rem:AandW}
\begin{itemize}
\item[(a)] It follows directly from the definition that 
$$
\|fg\|_{\Ell^{2,\infty}(X|Y)}\leq \|f\|_{\Ell^\infty(X)}\|g\|_{\Ell^{2,\infty}(X|Y)}
$$
holds for every $f\in\Ell^\infty(X)$ and $g\in\Ell^{2,\infty}(X|Y)$.
\item[(b)]
The space $\Ell^2(X|Y)$ is clearly an $\Ell^\infty(Y)$-module (i.e., a linear subspace closed under multiplication by $\Ell^\infty(Y)$-functions) with
$$
\Ell^\infty(X)\subset \Ell^2(X|Y)\subset \Ell^2(X).
$$
Moreover, $\Ell^2(X|Y)$ is $T$-invariant by \eqref{eq:ET=SE} with $T$ acting as a contraction w.r.t.~the norm $\|\cdot\|_{\Ell^{2,\infty}(X|Y)}$.
\item[(c)] It is easy to see that $A(X|Y)$ is a $T$- and $T^{-1}$-invariant sub-$\Ell^\infty(Y)$-module.  Moreover, it is closed in $\Ell^2(X|Y)$ w.r.t.~the norm $\|\cdot\|_{\Ell^{2,\infty}(X|Y)}$.
\item[(d)] It is easy but somewhat longer to show that $W(X|Y)$ is a closed linear subspace of $\Ell^2(X)$.\footnote{See Zorin-Kranich \cite[Lemma 8.1]{Pavel-ET} based on Tao \cite[Exercises 2.14.1, 2.14.2]{Tao-ET}.} 
\item[(e)] Analogously to Remark \ref{rem:wmix},
 $f\in \Ell^2(X|Y)$ is conditionally weakly mixing if and only if there exists a set $J\subset \N$ with density $1$ such that $\lim_{n\to\infty,n\in J}\|\la T^nf,f\ra_{\Ell^2(X|Y)}\|_{\Ell^1(Y)}=0$.  
\item[(f)] For the trivial one-point factor $(Y,\nu,S)$ the conditional definitions coincide with the unconditional ones. In particular, in this case $f\in W(X|Y)$ if and only if $f$ is weakly mixing.
\end{itemize}
\end{remark}

\begin{remark}[Connection to conditional almost periodicity in measure]\label{rem:capm}
A function $f\in \Ell^2(X|Y)$ is called  \emph{conditionally almost periodic in measure} if for every $\veps>0$ there exists a measurable set $M\subset Y$ with $\nu(M)>1-\veps$ such that $\Eins_Mf\in A(X|Y)$, see Tao \cite[Def.~2.13.7]{Tao-ET}, cf.~Furstenberg, Katznelson \cite[Section 2]{FuKa78} and Furstenberg \cite[Section 6.3]{Fu-book}. Denote the set of all such functions by $AM(X|Y)$. This set satisfies  
$$
AM(X|Y)=\ol{A(X|Y)}\cap \Ell^2(X|Y),
$$
where the closure is taken in $\Ell^2(X)$, see, e.g., Zorin-Kranich \cite{Pavel}, cf.~Furstenberg, Katznelson \cite[Thm.~2.1]{FuKa78} and Furstenberg \cite[Thm.~6.13]{Fu-book}. For the reader's convenience we present here the short argument. The inclusion ``$\subset$'' follows directly from the definition of conditional almost periodicity. For the converse inclusion let $f\in\Ell^2(X|Y)$ such that $f=\lim_{n\to\infty} f_n$ in $\Ell^2(X)$ for some sequence $(f_n)\subset A(X|Y)$. By \eqref{eq:EY} we have
$$
\int_Y \EY(|f_n-f|^2)\,d\nu=\|f_n-f\|_{\Ell^2(X)}^2\to 0\quad \text{as }n\to\infty.
$$
Hence we can assume without loss of generality (passing to a subsequence if ne\-cessary) that $\EY(|f_n-f|^2)\to 0$ as $n\to\infty$ holds $\nu$-a.e. By Egorov's theorem there exists a measurable set $M\subset Y$ with $\nu(M)>1-\veps$ such that $\|\Eins_M\EY(|f_n-f|^2)\|_{\Ell^\infty(Y)}\to 0$, or, equivalently by \eqref{eq:EY}, 
$$
\|\Eins_Mf_n-\Eins_Mf\|_{\Ell^{2,\infty}(X|Y)}\to 0\quad \text{as }n\to\infty.
$$
Since for every $n\in\N$ the function $\Eins_M f_n$ is conditionally almost periodic (by using  the same module zonotopes as for $f_n$), so is $\Eins_M f$ by Remark \ref{rem:AandW}(c). 
\end{remark}

The following conditional version of Theorem \ref{thm:JdLG} is due to Tao \cite[Section 2.14.2]{Tao-ET}, cf.~Furstenberg, Katznelson \cite[Prop.~2.2]{FuKa78}
and Furstenberg, Katznelson, Ornstein \cite[Thm.~10.1]{FuKaOr82}. 
Here and later we mean the closure in $\Ell^2(X)$ if not specified otherwise.

\begin{thm}[
Conditional Jacobs-de Leeuw-Glicksberg decomposition]\label{mr:higher-JdLG}
Let $(X,\mu,T)$ be a 
\mps\ with  
factor $(Y,\nu,S)$. Then 
one has the orthogonal decomposition
\begin{equation}\label{eq:rel-jdlg-a+w}
\Ell^2(X)=\ol{A(X|Y)}\oplus W(X|Y). 
\end{equation}
\end{thm}

\begin{remark}\label{rem:a=e}
Zorin-Kranich \cite{Pavel} showed that for regular 
systems with ergodic factor 
\begin{equation}\label{eq:A=E}
\ol{A(X|Y)}=\ol{E(X|Y)}
\end{equation}
holds for the set $E(X|Y)$ of conditional eigenfunctions, see  Jamneshan \cite[Theorem 4.1]{Ja} for the general case in the context of arbitrary group actions. 
Here  $f\in \Ell^2(X|Y)$ is called a \emph{conditional eigenfunction} if 
its orbit is contained in a $T$-invariant \emph{finitely generated sub-$\Ell^\infty(Y)$-module} (i.e., a set of the form 
$$
\{c_1f_1+\ldots+c_mf_m:\,c_1,\ldots,c_m\in\Ell^\infty(Y)\}
$$ for some $m\in\N$, $f_1,\ldots,f_m\in\Ell^2(X|Y)$).\footnote{This  definition due to Furstenberg, Weiss \cite{FW96}, see also Zimmer 
\cite[Section 7]{Zi}, is equivalent to the one of generalized eigenfunctions introduced by Furstenberg \cite{Fu77} for ergodic factors $(Y,\nu,S)$.} Thus we also have the orthogonal decomposition
\begin{equation}\label{eq:rel-jdlg-e+w}
\Ell^2(X)=\ol{E(X|Y)}\oplus W(X|Y),
\end{equation}
cf.~Zimmer \cite[Cor.~7.10]{Zi}.
For two different recent abstract approaches to 
\eqref{eq:rel-jdlg-e+w} in the context of general group actions 
see Jamneshan \cite{Ja} and Edeko, Haase, Kreidler \cite{EHK}. 
\end{remark}

%

The following lemma is crucial,  see Tao \cite[Exercise 2.13.6]{Tao-ET} combined with Remark \ref{rem:capm}. 
We assume ergodicity of the factor in order to simplify the argument and follow the proof by Zorin-Kranich \cite{Pavel}. 

\begin{lemma}\label{lem:mul-E-factor}
Let $(X,\mu,T)$ be a \mps\ with an ergodic factor $(Y,\nu,S)$. Then 
$\ol{A(X|Y)}$ is induced by a factor. 
\end{lemma}

\begin{proof}
We proceed in three steps.

\underline{Step 1.} We first show that for every $f\in A(X|Y)$ one can take \emph{bounded} generators of the module zopotopes 
in the definition of conditional almost periodicity.  Let $\veps>0$ and let $f_1,\ldots,f_m\in\Ell^2(X|Y)$ satisfy 
$$
\orb(f):=\{T^nf:\,n\in\Z\}\subset U_\veps(Z(f_1,\ldots,f_m)),
$$
where we denote by $Z(f_1,\ldots,f_m)$ the module zonotope generated by $f_1,\ldots,f_m$ and by $U_\veps(Z)$ the $\veps$-neighborhood of a set $Z\subset\Ell^2(X|Y)$ with respect to the norm $\|\cdot\|_{\Ell^{2,\infty}(X|Y)}$.

Fix $j\in\{1,\ldots,m\}$ and consider for every $k\in \N$ the function $f_{j,k}:=f_j\Eins_{\{|f_j|\leq k\}}$ which is bounded by $k$. Since  the functions $f_{j,k}$ approximate $f_j$ in $\Ell^2(X)$, $\EY (|f_j-f_{j,k}|^2)$ converge to zero in $\Ell^1(Y)$ as $k\to\infty$. By Egorov's theorem (passing first to a subsequence if necessary), we see that there exists $k\in \N$ and a set $A_j\subset Y$ with $\nu(A_j)>1-\frac{\veps}m$ such that $\EY(|f_j-f_{j,k}|^2)<\frac{\veps}m$ on $A_j$. Write now $f_j=g_j+b_j+r_j$
with $g_j:=f_{j,k}$ being bounded by $k$, $b_j:=(f_j-f_{j,k})\Eins_{Y\setminus A_j}$ and the rest $r_j:=(f_j-f_{j,k})\Eins_{A_j}$. Since $\|r_j\|_{\Ell^{2,\infty}(X|Y)}\leq \frac{\veps}m$ by the definition of $k$, we have
\begin{equation}\label{eq:U2eps}
\orb(f)\subset U_{2\veps} (Z(g_1,\ldots,g_m,b_1,\ldots,b_m)).
\end{equation}
Denote by $\varphi$ and $\psi$ the underlying transformations on $X$ and $Y$, respectively. 
Consider the ``good'' set  $B:=\cap_{j=1}^mA_j$ satisfying $\nu(B)>1-\veps$. By \eqref{eq:U2eps} and since the functions $b_j$  vanish on $B$, for every $l\in\N_{0}$ there exist $s_l\in Z(g_1,\ldots,g_m)$ (which is then bounded by $M:=km$) such that 
$$
\|1_B T^lf-s_l\|_{\Ell^{2,\infty}(X|Y)}<2\veps
$$
which we can rewrite as 
\begin{equation}\label{eq:S^lBf}
\|1_{\psi^lB} f-T^{-l}s_l\|_{\Ell^{2,\infty}(X|Y)}<2\veps.
\end{equation}
Consider now the function
$$
\tilde{f}= \sum_{l=0}^\infty \Eins_{\psi^l(B)}T^{-l}s_l \Eins_{Y\setminus(\cup_{l'=0}^{l-1}\psi^l(B))}.
$$
Clearly, $\tilde{f}$ is bounded by $M$. Moreover, since  $\cup_{l=0}^\infty \psi^l(B)=Y$ up to a null set by the ergodicity of $(Y,\nu,S)$, \eqref{eq:S^lBf} implies
\begin{equation}\label{eq:f-tildef}
\|f-\tilde{f}\|_{\Ell^{2,\infty}(X|Y)}\leq 2\veps.
\end{equation}
In particular, $\tilde{f}$ satisfies by \eqref{eq:U2eps}
\begin{equation}\label{eq:U4eps}
\orb(\tilde{f})\subset U_{4\veps} (Z(g_1,\ldots,g_m,b_1,\ldots,b_m)).
\end{equation}
We now truncate the 
new generators as follows. Let $N:=\frac{M}{\veps}\sum_{j=1}^m \|b_j\|_{\Ell^{2,\infty}(X|Y)}$ and set $D:=\cap_{j=1}^m \{|b_j|\leq N\}\subset X$. By $\Eins_{X\setminus D}\leq \sum_{j=1}^m\frac{|b_j|}N$ we have 
\begin{equation}\label{eq:norm-x-d}
\| \Eins_{X\setminus D}\|_{\Ell^{2,\infty}(X|Y)}\leq \frac1N \sum_{j=1}^m \|b_j\|_{\Ell^{2,\infty}(X|Y)}=\frac{\veps}M. 
\end{equation}
By \eqref{eq:U4eps} for every $j\in\{1,\ldots,m\}$ and $n\in\Z$  there exist $c_{j,n},d_{j,n}\in \Ell^\infty(Y)$ bounded by $1$ such that 
$$
\left\|T^n\tilde{f} - \sum_{j=1}^mc_{j,n}g_j- \sum_{j=1}^md_{j,n}b_j\right\|_{\Ell^{2,\infty}(X|Y)}<4\veps.
$$
It follows, by adding and subtracting $\Eins_DT^n\tilde{f}$ and using \eqref{eq:norm-x-d} and Remark \ref{rem:AandW}(a), that
\begin{eqnarray*}
&\ &\left\|T^n\tilde{f} - \sum_{j=1}^mc_{j,n}g_j\Eins_D- \sum_{j=1}^md_{j,n}b_j\Eins_D\right\|_{\Ell^{2,\infty}(X|Y)}\\
&\ &\quad \leq \left\|T^n\tilde{f} - \sum_{j=1}^mc_{j,n}g_j- \sum_{j=1}^md_{j,n}b_j\right\|_{\Ell^{2,\infty}(X|Y)}+\|T^n\tilde{f}\|_{\Ell^\infty(X)}\| \Eins_{X\setminus D}\|_{\Ell^{2,\infty}(X|Y)}\\
&\ &\quad < 4\veps + M\frac{\veps}M=5\veps.
\end{eqnarray*}
This implies by \eqref{eq:f-tildef}
$$
\orb(f)\subset U_{7\veps}(Z(g_1\Eins_D,\ldots,g_m\Eins_D,b_1\Eins_D,\ldots,b_m\Eins_D)),
$$
with every generator of this zonotope being bounded by $\max\{k,N\}$. This completes the proof of Step 1. 

\underline{Step 2.} We now prove that $A(X|Y)\cap \Ell^\infty(X)$ is dense in $\ol{A(X|Y)}$.
To do this it is enough to show that for every $f\in A(X|Y)$  the function  $f^+:=\max\{0,f\}$ belongs to $A(X|Y)$ which by shifting, multiplying with $-1$ and applying the argument again implies conditional almost periodicity of the function $f\Eins_{\{|f|\leq N\}}$ for every $N\in \N$. 
Let $f\in A(X|Y)$ and $\veps>0$. By Step 1 we can assume that the orbit of $f$ belongs to an $\veps$-neighborhood (w.r.t.~$\|\cdot\|_{\Ell^{2,\infty}(X|Y)}$) of a module zonotope with bounded generators $f_1,\ldots,f_m$. Using $\|\cdot\|_{\Ell^{2,\infty}(X|Y)}\leq \|\cdot\|_{\Ell^\infty(X)}$, by changing to a $2\veps$-neighborhood we can assume 
that each $f_j$ is a finite linear combination of characteristic functions. Moreover, by 
enlarging $m$ we can further assume that each $f_j$ is a positive multiple of a characteristic function with supports being disjoint (up to null sets). After these modifications, by 
$$
\left\|T^n f^+ - \sum_{j=1}^m c_j^+f_j\right\|_{\Ell^{2,\infty}(X|Y)}\leq \left\|T^n f - \sum_{j=1}^m c_jf_j\right\|_{\Ell^{2,\infty}(X|Y)}
$$
we see that the orbit of $f^+$ belongs to the same neighborhood of the same module zonotope as the orbit of $f$ finishing the argument. 

\underline{Step 3.} It is clear that the set $A(X|Y)\cap \Ell^\infty(X)$ is conjugation invariant, contains $\Eins$ and is invariant under both $T$ and $T^{-1}$. Moreover, by Step 2 it is dense in $\ol{A(X|Y)}$. It thus remains to show that it is a subalgebra of $\Ell^\infty(X)$. Let $f,g\in A(X|Y)\cap \Ell^\infty(X)$ and assume without loss of generality that $\|g\|_\infty\leq 1$. Let further $\veps>0$ and  take $f_1,\ldots, f_m\in\Ell^2(X|Y)$ with $\orb(f)\in U_\frac{\veps}{2}(Z(f_1,\ldots,f_m))$. By Step 1 we can assume that $f_1,\ldots,f_m\in \Ell^\infty(X)$ and denote $M:=\max\{\|f_1\|_\infty,\ldots,\|f_m\|_\infty\}$. Let further $g_1,\ldots,g_l\in\Ell^2(X|Y)$ satisfy $\orb(g)\in U_{\frac{\veps}{2mM}}(Z(g_1,\ldots,g_l))$. Then for every $n\in\Z$ and appropriate $a_n\in Z(f_1,\ldots,f_m)$ and $b_n\in Z(g_1,\ldots,g_l)$ we have by the triangle inequality and Remark \ref{rem:AandW}(a) 
\begin{eqnarray*}
&\ &\|T^n(fg)-a_nb_n
\|_{\Ell^{2,\infty}(X|Y)}\leq \|T^nf-a_n\|_{\Ell^{2,\infty}(X|Y)}\|g\|_{\Ell^\infty(X)}\\
&\ &\quad +\|a_n\|_{\Ell^\infty(X)}\|T^ng-b_n \|_{\Ell^{2,\infty}(X|Y)}<\frac{\veps}2+mM\frac{\veps}{2mM}=\veps.
\end{eqnarray*}
Since each $a_nb_n$ belongs to the module zonotope generated by $f_jg_k$, $j\in\{1,\ldots,m\}$, $k\in\{1,\ldots,l\}$, this shows $fg\in A(X|Y)\cap \Ell^\infty(X)$. The proof is complete. 
\end{proof}

\begin{remark}
The factor inducing $\ol{A(X|Y)}$ is the maximal compact extension
of $Y$ in $X$ in the sense of Definition \ref{def:comp-ext} below. Moreover, by \eqref{eq:A=E} it coincides with the maximal isometric extension of $Y$ in $X$  introduced by Furstenberg, Weiss \cite{FW96} for ergodic systems.
\end{remark}

As a corollary we obtain the following property of $W(X|Y)$.
\begin{lemma}\label{lem:mul-Wbdd}
Let $(X,\mu,T)$ be a 
\mps\ with 
factor $(Y,\nu,S)$. Then $W(X|Y)\cap\Ell^\infty(X)$ is dense in $W(X|Y)$ w.r.t.~the $\Ell^2$-norm.
\end{lemma} 
\begin{proof}
Let $f\in W(X|Y)$ and take a sequence $(f_n)\subset \Ell^\infty(X)$ converging to $f$ in $\Ell^2(X)$. Consider the sequence $(P_{W(X|Y)}f_n)$
where $P_{W(X|Y)}$ denotes the orthogonal projection onto $W(X|Y)$. This sequence clearly belongs to $W(X|Y)$ and converges to $f$ by the Pythagoras theorem. Moreover, for every $n\in\N$ we have
$$
P_{W(X|Y)}f_n=f_n-P_{\ol{A(X|Y)}}f_n\in \Ell^\infty(X)
$$
by \eqref{eq:rel-jdlg-a+w} and Lemma \ref{lem:mul-E-factor}. (Recall that 
the projection onto a factor is 
contractive w.r.t.~the  $\Ell^\infty$-norm.)
\end{proof}

We now introduce the distal factors of finite order defined by Furstenberg \cite{Fu77} using conditional eigenfunctions instead of conditionally almost periodic functions (cf.~Remark \ref{rem:a=e}). 

\begin{definition}\label{def:mul-n}
Let $(X,\mu,T)$ be an ergodic 
\mps. We construct the sequence of factors inductively as follows. Start with $D_0:= \fix T=\C\Eins$ and the corresponding fixed (one-point) factor $\mathcal{D}_0$ 
and for every $k\in\N$ denote by $\mathcal{D}_{k}$ the factor 
inducing $D_k:=\ol{A(X|\mathcal{D}_{k-1})}$ 
(see Lemma \ref{lem:mul-E-factor}). We call $\mathcal{D}_k$ the \emph{distal factor of order $k$} of $(X,\mu,T)$.
\end{definition}

Furstenberg \cite{Fu77} showed that for regular, ergodic systems  for each $k\in\N$ the factor $\mathcal{D}_k$ (defined via conditional eigenfunctions) is \emph{characteristic for $k$-term multiple ergodic averages \eqref{eq:mult-erg-ave}} in the sense that both convergence and the limit of the averages \eqref{eq:mult-erg-ave} remain unchanged if we replace every function by its projection
onto this factor. See also Frantzikinakis \cite[Theorem 5.2]{Frantz04} who deduced it from Furstenberg \cite[Theorem 7.1]{Fu77} using the van der Corput trick. We will give an alternative proof of this fact here, see Propositions \ref{prop:ghk-control-mult-ave} and  \ref{prop:Z_k-ap} below.


\section{Gowers-Host-Kra seminorms}\label{sec:GHK-seminorms} 

The uniformity seminorms were introduced by Gowers \cite{Gowers} in his proof of Szemer\'edi's theorem via higher order Fourier analysis for rotations on cyclic gro\-ups $\Z_N$
and were extended by Host and Kra \cite{HK}\footnote{Initially Host, Kra \cite{HK} defined the seminorms using cube measure spaces. The two definitions are easily shown to be equivalent, see, e.g., Kra \cite[Lemma 7.4]{Kra} or \cite[Section 14.2]{EF-book}.} to arbitrary ergodic measure-preserving systems in their proof of multiple convergence.

\begin{definition}[Gowers-Host-Kra (uniformity) seminorms]
\label{def:GHK-1}
Let $(X,\mu,T)$ be an ergodic 
 \mps \ and $f\in L^\infty(X,\mu)$. The \emph{Gowers-Host-Kra} (or \emph{uniformity}) \emph{seminorms} are defined inductively by
\begin{eqnarray}
\|f\|_{U_1}&:=&
\left| \int_X f\,d\mu\right|
,\label{eq:u1}\\
\|f\|_{U_{l+1}}^{2^{l+1}}&:=&\limsup_{N\to\infty}\aveN \|T^nf\cdot \ol{f}\|_{U_l}^{2^l},\quad l\in\N.\label{eq:ghk-def-seminorms}
\end{eqnarray}
\end{definition}

\begin{remark}[Second uniformity seminorm]\label{rem:semin-erg-12}
For an ergodic measure-preserving system $(X,\mu,T)$ the second uniformity seminorm satisfies
\begin{eqnarray*}
\|f\|_{U_2}^4&=&\limsup_{N\to\infty}\aveN |\la T^nf,f\ra|^{2}.
\end{eqnarray*}
In particular,  $f\in\Ell^\infty(X)$ is weakly mixing if and only if $\|f\|_{U^2}=0$. 
\end{remark}


It is easy to see that the  Gowers-Host-Kra  seminorms are increasing and satisfy 
\begin{equation}\label{eq:U-Linfty}
\|\cdot\|_{U^k}\leq \|\cdot\|_{\Ell^\infty(X)}\quad\forall k\in\N.
\end{equation}
We need the following stronger property of these seminorms, see, e.g., \cite[Equation (12)]{ET}.

\begin{lemma}\label{lem:mul-uk-lp}
Let $(X,\mu,T)$ be an ergodic \mps\ and $k\in \N$. Then for $p_k:=\frac{2^k}{k+1}$ one has $\|f\|_{U^k}\leq \|f\|_{\Ell^{p_k}(X)}$ for every $f\in\Ell^\infty(X)$. 
\end{lemma}

Let $(X,\mu,T)$ be an ergodic \mps. 
Consider for every $k\in \N_{0}$ the orthogonal decomposition
\begin{equation}\label{eq:hk-dec}
\Ell^2(X)=Z_k\oplus \ol{\{f\in\Ell^\infty(X):\ \|f\|_{U^{k+1}}=0\}}
\end{equation}
known as the \emph{Host-Kra decomposition} of order $k$, where $Z_k$ is at first defined as the orthogonal complement of the second part. In particular, we have $Z_0=\C\Eins=\Fix T$ by ergodicity and the decomposition \eqref{eq:hk-dec} for $k=0$ coincides with the von Neumann decomposition \eqref{eq:mul-vN-dec}.
Moreover, for $k=1$ \eqref{eq:hk-dec} coincides with the 
Jacobs-de Leeuw-Glicksberg decomposition \eqref{eq:jdlg-dec} by 
Remark \ref{rem:semin-erg-12}
and Lemma \ref{lem:mul-Wbdd} applied to the one-point factor. 

The following property shows the relevance of the uniformity seminorms for multiple convergence and recurrence. It is an easy consequence of the definition of the Gowers-Host-Kra seminorms and the van der Corput inequality, see, e.g., \cite[Section 14.1]{EF-book}. 

\begin{proposition}[Generalized von Neumann theorem]
\label{prop:ghk-control-mult-ave}
Let $(X,\mu,T)$ be an ergodic measure-preserving system, $k\in\N$ and $f_1,\ldots,f_k\in\Ell^\infty(X)$. Then 
\begin{equation}\label{eq:gener-vN-ineq}
\limN 
\left\|\aveN T^nf_1\cdot T^{2n}f_2\cdots T^{kn}f_k\right\|_2\leq \min_{j=1,\ldots,k}j\|f_j\|_{U^k}.
\end{equation}
\end{proposition}

\begin{corollary}[The subspace $Z_{k-1}$ is characteristic for $k$-term multiple averages]\label{cor:ghk-control-mult-ave}
Let $(X,\mu,T)$ be an ergodic measure-preserving system, $k\in\N$ and $f_1,\ldots,f_k\in\Ell^\infty(X)$. Then 
$$
\limaveN T^nf_1\cdot T^{2n}f_2\cdots T^{kn}f_k=0 \quad \text{in }\Ell^2(X)
$$
holds whenever $f_j\perp Z_{k-1}$ for some $j\in\{1,\ldots,k\}$.
\end{corollary}
\begin{proof}
Assume without loss of generality that $\|f_1\|_\infty\leq 1,\ldots,\|f_k\|_\infty\leq 1$ and let $f_j\perp Z_{k-1}$ for some $j\in\{1,\ldots,k\}$. By the Host-Kra decomposition \eqref{eq:hk-dec} there exists a sequence $(g_m)\subset \Ell^\infty(X)$ satisfying $\|g_m\|_{U^k}=0$ for all $m\in\N$ with $\lim_{m\to\infty}\|f_j-g_m\|_2=0$. Thus for every $m\in\N$, by decomposing $f_j=(f_j-g_m)+g_m$, we have by \eqref{eq:gener-vN-ineq} and the triangle inequality 
\begin{eqnarray*}
\limsup_{N\to\infty}\left\|\aveN T^nf_1\cdot T^{2n}f_2\cdots T^{kn}f_k\right\|_2\leq \|f_j-g_m\|_{\Ell^2(X)}.
\end{eqnarray*}
Letting $m\to\infty$ proves the assertion.
\end{proof}


The following property of the subspaces $Z_k$ can be shown via constructing the factors using cube measure spaces, see Host, Kra \cite{HK},  \cite[Chapter 9]{HK-book} and, e.g., \cite[Section 14.4]{EF-book}. We will not need this fact and include it here for completeness.

\begin{proposition}\label{prop:HKZ-factors}
For every $k\in\N_{0}$ the subspace $Z_k$ is induced by a factor 
 $\mathcal{Z}_k$, called the \emph{Host-Kra-Ziegler factor of order $k$}. 
\end{proposition}

In particular, $\mathcal{Z}_0$ is the one-point factor and $\mathcal{Z}_1$ is the Kronecker factor. 
The Host-Kra-Ziegler factors 
were introduced by Host and Kra \cite{HK} and subsequently independently by Ziegler \cite{Ziegler} (see Leibman \cite{Le} for the equality of the two constructions) in their proofs for multiple convergence. 
We refer to Host, Kra \cite{HK,HK-book} for a detailed analysis of these factors and the deep structure theorem which states that for ergodic regular systems each $\mathcal{Z}_k$ is an inverse limit of nilsystems of step $k$. Again, we will not need anything from this theory 
here.

\begin{remark}[Multiple recurrence for weakly mixing systems]
For weakly mixing systems\footnote{A \mps\ $(X,\mu,\varphi)$ is called \emph{weakly mixing} if $\limaveN |\mu(A\cap \varphi^{-n}B)-\mu(A)\mu(B)|=0$ for every measurable $A,B\subset X$. This is equivalent to the orthogonal decomposition $\Ell^2(X)=\C\Eins\oplus \{\text{weakly mixing functions}\}$.} one easily shows using Remark \ref{rem:semin-erg-12} and induction  that  all Gowers-Host-Kra seminorms are equal to the first seminorm (see, e.g., Kra \cite[Section 7.3]{Kra}). Thus in this case $Z_k=\C\Eins$ holds for every $k\in\N$ 
and both multiple convergence and multiple recurrence follow from Corollary \ref{cor:ghk-control-mult-ave}
with 
\begin{equation*}
\lim_{N\to\infty}\aveN \int_X f_0\cdot T^nf_1\cdot T^{2n}f_2 \cdots T^{kn}f_k\,d\mu=\prod_{j=0}^k\int_Xf_j\,d\mu
\end{equation*}
recovering a result of Furstenberg \cite[Equation (1)]{Fu77}.
\end{remark}

\section{Host-Kra-Ziegler factors versus distal factors \\
and end of the proof}


We now connect the Host-Kra-Ziegler factors to distal factors of the corresponding order using the conditional Jacobs-de Leeuw-Glicksberg decomposition (Theorem \ref{mr:higher-JdLG}) and Lemma \ref{lem:mul-uk-lp}.

\begin{proposition}[
Distal factors are extensions of Host-Kra-Ziegler factors]\label{prop:Z_k-ap}
Let $(X,\mu,T)$ be an ergodic \mps. 
Then 
$Z_{k}\subset D_{k}$ for every $k\in \N$. 
\end{proposition}
\begin{proof}
Let $k\in \N_{0}$.
%
By the orthogonal decomposition 
\begin{equation}\label{eq:mul-dec-dk+1}
\Ell^2(X)=D_{k+1}\oplus W(X|\DD_{k})
\end{equation}
 from the conditional \jdlg\ Theorem \ref{mr:higher-JdLG} and the Host-Kra decomposition \eqref{eq:hk-dec}
 for $k+1$ we need to show 
$$
W(X|\DD_{k})\subset \ol{\{f\in\Ell^\infty(X):\, \|f\|_{U^{k+2}}=0\}}.
$$
By Lemma \ref{lem:mul-Wbdd} it suffices to show 
\begin{equation}\label{eq:mul-uk+2=0}
\text{$\|f\|_{U^{k+2}}=0$  for every 
$f\in W(X|\DD_{k})\cap \Ell^\infty(X)$. }
\end{equation}
We show \eqref{eq:mul-uk+2=0} for every $k\in\N_{0}$ by induction on $k$. For $k=0$ let $f\in W(X|\mathcal{D}_0)\cap \Ell^\infty(X)$, where $\mathcal{D}_0$ is the one-point factor. Then $f$ is weakly mixing by Remark \ref{rem:AandW}(f) and hence $\|f\|_{U^2}=0$ by Remark \ref{rem:semin-erg-12}.
%
Assume now that $k\in \N$ and that \eqref{eq:mul-uk+2=0} holds for $k-1$. 
 Let $f\in W(X|\DD_{k})$ be bounded and assume without loss of generality $\|f\|_\infty\leq 1$. Take $n\in \N$. By the triangle inequality  and the decomposition \eqref{eq:mul-dec-dk+1} 
for $k-1$ 
\begin{equation}\label{eq:mul-uk-triangle}
\|T^nf\cdot\ol{f}\|_{U^{k+1}}\leq \|P_{D_{k}}(T^nf\cdot\ol{f})\|_{U^{k+1}} + \|P_{W(X|\DD_{k-1})}(T^nf\cdot\ol{f})\|_{U^{k+1}}
\end{equation}
for the orthogonal projections $P_{D_{k}}$ and $P_{W(X|\DD_{k-1})}$ onto the factor $\DD_{k}$ and the subspace $W(X|\DD_{k-1})$, respectively. 
(Recall that the first projection and hence also the second, complementary, projection 
maps bounded functions to bounded functions.) 
The last summand in \eqref{eq:mul-uk-triangle} equals zero by the induction hypothesis. Moreover, recall that $P_{D_k}=JE_{\DD_k}$ for the corresponding Markov factor map $J$ and that both $J$ and $E_{\DD_k}$ act as contractions w.r.t.~the $\Ell^p$-norm for every $p\in [1,\infty]$.  
So we have by \eqref{eq:mul-uk-triangle}, Lemma \ref{lem:mul-uk-lp} and $\|f\|_\infty\leq 1$ denoting $p_{k+1}:=\frac{2^{k+1}}{k+2}$
\begin{eqnarray*}
&\ &
\|T^nf\cdot\ol{f}\|_{U^{k+1}}
\leq 
\|P_{D_{k}}(T^nf\cdot\ol{f})\|_{U^{k+1}}
\leq 
\|P_{D_{k}}(T^nf\cdot\ol{f})\|_{\Ell^{p_{k+1}}(X)}
\\
&\ &\quad \leq 
\|E_{\DD_{k}}(T^nf\cdot\ol{f})\|_{\Ell^{p_{k+1}}(Y)}
\leq \|E_{\DD_{k}}(T^nf\cdot\ol{f})\|_{\Ell^1(Y)}^{1/p_{k+1}}.
\end{eqnarray*}
Since $\|T^nf\cdot\ol{f}\|_{U^{k+1}}\leq 1$ by $\|f\|_\infty\leq 1$ and \eqref{eq:U-Linfty},  this implies 
$$
\|T^nf\cdot\ol{f}\|_{U^{k+1}}^{2^{k+1}}\leq \|T^nf\cdot\ol{f}\|_{U^{k+1}}^{p_{k+1}}\leq \|E_{\DD_{k}}(T^nf\cdot\ol{f})\|_{\Ell^1(Y)}
$$
and, by $f\in W(X|\DD_k)$, 
$$
\|f\|_{U^{k+2}}^{2^{k+2}}=\limsup_{N\to\infty}\aveN \|T^nf\cdot\ol{f}\|_{U^{k+1}}^{2^{k+1}}=0
$$
follows.
This proves \eqref{eq:mul-uk+2=0} and completes the proof. 
\end{proof}

\begin{remark}\label{rem:Fu-orig}
As mentioned in the introduction, as a corollary of  Proposition \ref{prop:ghk-control-mult-ave} and Proposition \ref{prop:Z_k-ap} we obtain an alternative proof of the fact proved by Furstenberg \cite{Fu77}, see also Frantzikinakis \cite[Theorem 5.2]{Frantz04} based on Furstenberg \cite[Theorem 7.1]{Fu77}, that the distal factor $\DD_{k-1}$ of order $k-1$ is characteristic for $k$-term multiple recurrence. 
\end{remark}

We now define compact extensions, cf.~Furstenberg, Katznelson \cite[Def.~3.1]{FuKa78} and Furstenberg, Katznelson, Ornstein \cite[Section 9]{FuKaOr82}.
\begin{definition}\label{def:comp-ext}
A \mps\ $(X,\mu,T)$ is called a \emph{compact extension}\footnote{Tao \cite[Def.~2.13.7]{Tao-ET} calls an extension $(X,\mu,T)$ of $(Y,\nu,S)$ compact if every function $f\in\Ell^2(X|Y)$ is conditionally almost periodic in measure, cf.~Furstenberg, Katznelson \cite[Section 2]{FuKa78}, Furstenberg \cite[Section 6.3]{Fu-book}. By Remark \ref{rem:capm} the two definitions are equivalent.} of its factor $(Y,\nu,S)$ if $\ol{A(X|Y)}=\Ell^2(X)$. 
\end{definition}
%
In particular, for an ergodic \mps\  $(X,\mu,T)$ and every $k\in \N_{0}$ the distal factor $\mathcal{D}_{k+1}$ is by definition 
a compact extension of $\mathcal{D}_{k}$, namely 
 the maximal compact extension of $\mathcal{D}_{k}$ in $(X,\mu,T)$.

The last ingredient of the proof of Theorem \ref{thm:mr} is the following property of compact extensions, see Tao \cite[Theorem 2.13.11]{Tao-ET} combined with Remark \ref{rem:capm}, cf.~Furstenberg, Katznelson \cite[Prop.~3.5]{FuKa78}, 
Furstenberg, Katznelson, Ornstein \cite[Thm.~9.1]{FuKaOr82} and Einsiedler, Ward \cite[Section 7.9]{EW}. 

\begin{proposition}[Compact extensions preserve MR]\label{prop:ap-ext-MR}
Let $(X,\mu,T)$ be a compact extension of $(Y,\nu,S)$. 
If $(Y,\nu,S)$ satisfies MR then so does $(X,\mu,T)$.
\end{proposition}

\begin{proof}[Proof of multiple recurrence for ergodic systems (Theorem \ref{thm:mr})]
Let $(X,\mu,T)$ be an ergodic \mps\ and $k\in\N$.
Denoting by $P_{D_{k-1}}$ the orthogonal projection onto $D_{k-1}$, for $f\in\Ell^\infty(X)$ with $f>0$ the function $P_{D_{k-1}}f$ is also bounded and satisfies $P_{D_{k-1}}f>0$  by Lemma \ref{lem:mul-E-factor} and \eqref{eq:PYf>0}. 
Therefore Proposition \ref{prop:Z_k-ap} and Corollary \ref{cor:ghk-control-mult-ave} imply
\begin{equation}\label{eq:mul-mea-char-fac}
\limN \left\|\aveN T^nf\cdots T^{kn}f-\aveN T^nP_{D_{k-1}}f\cdots T^{kn}P_{D_{k-1}}f
\right\|_{\Ell^2(X)}=0
\end{equation}
and we can assume without loss of generality that $\Ell^2(X)=D_{k-1}$. 
%
So we have the chain of factors 
$$
X=
\DD_{k-1} \to \DD_{k-2} \to \ldots \to \DD_1 \to \DD_0=\{\cdot\}
$$
where $\{\cdot\}$ denotes the one-point factor, 
and every factor in this chain is a compact extension of the next one by definition. Finally, the one-point factor clearly satisfies MR. Thus every $\DD_j$, and hence also $X=\DD_{k-1}$,  satisfies MR by Proposition \ref{prop:ap-ext-MR}. 
The proof is complete.
\end{proof}



\begin{thebibliography}{10}



\bibitem{BL96}
V.~Bergelson, A.~Leibman, \emph{Polynomial extensions of van der Waerden's and Szemer\'edi's theorems}, J. Amer. Math. Soc. \textbf{9} (1996), 725--753.

\bibitem{BLL08}
V.~Bergelson, A.~Leibman, E.~Lesigne, 
\emph{Intersective polynomials and the polynomial Szemer\'edi theorem},
Adv. Math. \textbf{219} (2008), 369--388.

\bibitem{BM00}
V.~Bergelson, R.~McCutcheon, \emph{An ergodic IP polynomial Szemerédi theorem}, Mem. Am. Math. Soc. \textbf{146} (695) (2000), viii+106.

\bibitem{BMR20}
V.~Bergelson,  J.~Moreira, F.~K.~Richter, \emph{Single and multiple recurrence along non-polynomial sequences}, Adv. Math. \textbf{368} (2020), 107146, 69 pp. 

\bibitem{CMT}
B.~Cook, \'A.~Magyar, T.~Titichetrakun, \emph{A multidimensional Szemer\'edi theorem in the primes via combinatorics}, Ann. Comb. \textbf{22} (2018),  711--768. 

\bibitem{EHK}
N.~Edeko, M.~Haase, H.~Kreidler, \emph{A decomposition theorem for unitary group representations on Kaplansky-Hilbert modules and the Furstenberg-Zimmer structure theorem}, preprint, 2021, available at https://arxiv.org/abs/2104.04865.

\bibitem{EW} M.~Einsiedler, T.~Ward, Ergodic Theory with a View Towards Number Theory. Graduate Texts in Mathematics, 259. Springer-Verlag London, Ltd., London, 2011.

\bibitem{EF-book}
T.~Eisner, B.~Farkas, Ergodic Theorems. Book manuscript, submitted.

\bibitem{EFHN}
T.~Eisner, B.~Farkas, M.~Haase, R.~Nagel, Operator Theoretic Aspects of Ergodic Theory. Graduate Texts in Mathematics, vol. 272, Springer, Cham, 2015.

\bibitem{ET}
T.~Eisner,  T.~Tao, \emph{Large values of the Gowers-Host-Kra seminorms}, J. Anal. Math. \textbf{117} (2012), 133--186.

\bibitem{FoZh}
J.~Fox, Y.~Zhao, \emph{A short proof of the multidimensional Szemer\'edi theorem in the primes}, Amer. J. Math. \textbf{137} (2015), 1139--1145.

\bibitem{Frantz04}
N.~Frantzikinakis, \emph{The structure of strongly stationary systems}, J. Anal. Math. \textbf{93} (2004), 359--388.

\bibitem{Frantz15}
N.~Frantzikinakis, \emph{A multidimensional Szemer\'edi theorem for Hardy sequences of different growth}, Trans. Am. Math. Soc. \textbf{367} (8) (2015), 5653--5692. 

\bibitem{FHK07}
N.~Frantzikinakis, B.~Host, B.~Kra, \emph{Multiple recurrence and convergence for sequences related to the prime numbers}, J. Reine Angew. Math. \textbf{611} (2007), 131--144.


\bibitem{FHK13}
N.~Frantzikinakis, B.~Host, B.~Kra, \emph{The polynomial multidimensional Szemer\'edi theorem along shifted primes}, Israel J. Math. \textbf{194} (2013),  331--348.


\bibitem{FrWi09}
N.~Frantzikinakis, M.~Wierdl, \emph{A Hardy field extension of Szemer\'edi's theorem}, Adv. Math. \textbf{222} (2009),  1--43. 

\bibitem{Fu77}
H.~Furstenberg, \emph{Ergodic behavior of diagonal measures and a theorem of Szemer\'edi on arithmetic progressions},
J. Analyse Math. \textbf{31} (1977), 204--256. 

\bibitem{Fu-book}
H.~Furstenberg,  Recurrence in Ergodic Theory and Combinatorial Number Theory. M.~B.~Porter Lectures. Princeton University Press, Princeton, N.J., 1981.


\bibitem{FuKa78}
H.~Furstenberg, Y.~Katznelson, \emph{An ergodic Szemer\'edi theorem for commuting transformations}, J. Analyse Math. \textbf{34} (1978), 275--291 (1979).

\bibitem{FuKaOr82}
H.~Furstenberg, Y.~Katznelson, D.~Ornstein, 
\emph{The ergodic theoretical proof of Szemer\'edi's theorem},
Bull. Amer. Math. Soc. (N.S.) \textbf{7} (1982),  527--552. 


\bibitem{FW96}
H.~Furstenberg, B.~Weiss, 
\emph{A mean ergodic theorem for $(1/N)\sum_{n=1}^Nf(T^nx)g(T^{n^2}x)$},
Convergence in ergodic theory and probability (Columbus, OH, 1993), 193--227,
Ohio State Univ. Math. Res. Inst. Publ., 5, de Gruyter, Berlin, 1996.


\bibitem{Glasner-book} E.~Glasner, Ergodic Theory via Joinings. Mathematical Surveys and Monographs, 101. American Mathematical Society, Providence, RI, 2003.

\bibitem{GT08}
B.~Green, T.~Tao, \emph{The primes contain arbitrarily long arithmetic progressions} Ann. of Math. (2) \textbf{167} (2008), 481--547.

\bibitem{GT10}
B.~Green, T.~Tao, \emph{Linear equations in primes}, Ann. of Math. (2) \textbf{171} (2010), 1753--1850.

\bibitem{Gowers}
W.~T.~Gowers, \emph{A new proof of Szemer\'edi's theorem}, Geom. Funct. Anal. \textbf{11} (2001), 465--588.

\bibitem{HK}
B.~Host,  B.~Kra,  \emph{Nonconventional ergodic averages and nilmanifolds}, Ann. of Math. (2) \textbf{161} (2005), 397--488.

\bibitem{HK-book}
B.~Host,  B.~Kra, Nilpotent Structures in Ergodic Theory.
Mathematical Surveys and Monographs, 236. American Mathematical Society, Providence, RI, 2018. 

\bibitem{Ja}
A.~Jamneshan,  \emph{An uncountable Furstenberg-Zimmer structure theory}, preprint, 2021, available at https://arxiv.org/abs/2103.17167.

\bibitem{Kra}
B.~Kra, \emph{Ergodic methods in additive combinatorics}, Additive combinatorics, 103--143, CRM Proc. Lecture Notes, 43, Amer. Math. Soc., Providence, RI, 2007.

\bibitem{Le}
A.~Leibman, \emph{Host-Kra and Ziegler factors and convergence of multiple averages}, Handbook of Dynamical Systems, vol. 1B, B.~Hasselblatt and A.~Katok, eds., Elsevier (2005), 841--853.

\bibitem{MR19}
J.~Moreira, F.~K.~Richter,
\emph{A spectral refinement of the Bergelson-Host-Kra decomposition and new multiple ergodic theorems},
Ergodic Theory Dynam. Systems \textbf{39} (2019), 1042--1070.

\bibitem{Szem}
E.~Szemer\'edi, \emph{On sets of integers containing no $k$ elements in arithmetic progression}. Acta Arith. \textbf{27} (1975), 199--245.

\bibitem{Tao-szem}
T.~Tao, \emph{A quantitative ergodic theory proof of Szemer\'edi's theorem}, 
Electron. J. Combin. \textbf{13} (2006), no. 1, Research Paper 99, 49 pp. 

\bibitem{Tao-ET}
T.~Tao, Poincar\'e's legacies, pages from year two of a mathematical blog. Part II. American Mathematical Society, Providence, RI, 2009.

\bibitem{TZ08}
T.~Tao, T.~Ziegler, \emph{The primes contain arbitrarily long polynomial progressions}, Acta Math. \textbf{201} (2008), 213--305. 

\bibitem{TZ13}
T.~Tao, T.~Ziegler, \emph{Erratum to ``The primes contain arbitrarily long polynomial progressions''}, Acta Math. \textbf{210} (2013),  403--404.

\bibitem{TaZi15}
T.~Tao, T.~Ziegler, \emph{A multi-dimensional Szemerédi theorem for the primes via a correspondence principle},
Israel J. Math. \textbf{207} (2015), 203--228. 

\bibitem{TZ18}
T.~Tao, T.~Ziegler, \emph{Polynomial patterns in the primes}, Forum Math. Pi \textbf{6} (2018), e1, 60 pp.

\bibitem{WoZi12} T.~D.~Wooley, T.~Ziegler, \emph{Multiple recurrence and convergence along the primes}, Amer. J. Math. \textbf{134} (2012), 1705--1732.

\bibitem{Ziegler}
T.~Ziegler,  \emph{Universal characteristic factors and Furstenberg averages}, J. Amer. Math. Soc. \textbf{20} (2007), 53--97.

\bibitem{Pavel}  P.~Zorin-Kranich,  \emph{Compact extensions are isometric}, unpublished note, 2011. Available at https://www.math.uni-bonn.de/$\sim$pzorin/.

\bibitem{Pavel-ET} P.~Zorin-Kranich, Ergodic Theory, lecture notes, 2015/16. Available at https://www.math.uni-bonn.de/$\sim$pzorin/.

\bibitem{Zi}
R.~J.~Zimmer, \emph{Ergodic actions with generalized discrete spectrum}, Illinois J. Math. \textbf{20} (1976), 555--588.


\end{thebibliography}
\end{document}